\documentclass[12pt, a4]{amsart}
\usepackage{amscd, amsmath, amssymb}
\usepackage{fullpage}

\theoremstyle{plain}
\newtheorem{thm}{Theorem}[section]

\newtheorem{pro}[thm]{Proposition}
\theoremstyle{definition}
\newtheorem{defn}[thm]{Definition}

\newtheorem{rem}[thm]{Remark}
\newtheorem{cor}[thm]{Corollary}
\numberwithin{equation}{section}

\newcommand{\R}{\mathbb{R}}
\newcommand{\N}{\mathbb{N}}

\begin{document}

\title[Positive solutions for elliptic systems]{A non-variational approach to the existence of nonzero positive solutions for elliptic systems}  

\date{}

\author[J. {\'A}. Cid]{Jos\'e \'Angel Cid}
\address{Jos{\'e} {\'A}ngel Cid, Departamento de Matem\'aticas, Universidade de Vigo, 32004, Pabell\'on 3, Campus de Ourense, Spain}%
\email{angelcid@uvigo.es}%

\author[G. Infante]{Gennaro Infante}
\address{Gennaro Infante, Dipartimento di Matematica e Informatica, Universit\`{a} della
Calabria, 87036 Arcavacata di Rende, Cosenza, Italy}%
\email{gennaro.infante@unical.it}%

\begin{abstract} 
We provide new results on the existence of nonzero positive weak solutions for a class of second order elliptic systems. Our approach relies on a combined use of iterative techniques and classical fixed point index. Some examples are presented to illustrate the theoretical results.
\end{abstract}



\maketitle

\section{Introduction}
In the recent papers~\cite{lan1, lan2} Lan, by means of classical fixed point index theory, studied under \emph{sublinear} conditions the existence of nonzero positive solutions of systems of second order elliptic boundary value problems of the kind 
\begin{equation}
  \label{sys-intro}
 \left\{
\begin{array}{ll}
   Lu_i(x)= f_i(x,u(x)), &  x\in \Omega,\quad i=1,2,\ldots,n, \\
 Bu_i(x)=0, & x\in \partial \Omega,\quad i=1,2,\ldots,n,
\end{array}
\right.
\end{equation}
where $\Omega\subset \R^m$ is a suitable bounded domain,  $L$ is a strongly uniformly elliptic operator, $u(x)=(u_1(x),u_2(x),\ldots,u_n(x))$, $f_i:\bar{\Omega}\times  \R^m_+\to \R_+$ are continuous functions and $B$ is a first order boundary operator. 
A key assumption utilized by Lan is an asymptotic comparison between the nonlinearities and the principal eigenvalue of the corresponding linear system.

The formulation~\eqref{sys-intro} is rather general and covers systems of quasilinear elliptic equations of the type
\begin{equation}\label{lap-sys}
\left\{
\begin{array}{ll}
-\Delta u_1(x)=\lambda_1 f_1(x, u_1(x), u_2(x)), & x\in \Omega, \\
-\Delta u_2(x)=\lambda_2 f_2(x, u_1(x), u_2(x)), & x\in \Omega, \\
u_1(x)=u_2(x)=0, & x\in \partial \Omega,%
\end{array}%
\right.  
\end{equation}%
where $\lambda_i>0$. Elliptic systems of the kind~\eqref{lap-sys} have been widely investigated, by variational and topological methods,
under a variety of growth conditions; we refer the reader to the reviews by de~Figuereido~\cite{defig} and Ruf~\cite{bruf}, the papers~\cite{chzh, chzh07, cui3, defigetal, gimmrp, ma2, zhch} and references therein. 

Within the topological framework, fixed point index plays a key role and, typically, the methodology is to compute the index on a small ball and on a large ball, which, in turn, involves controlling the behaviour of an associated operator on two boundaries, see for example~\cite{alves-defig, defigetal, defig, lan1, lan-zhang, zou}.  
In the context of ODEs Franco and co-authors~\cite{df-gi-jp-prse} motivated by earlier works of Persson~\cite{persson} and continuing the work by Cabada and Cid~\cite{cabcid} and Cid et al.~\cite{jc-df-fm}, managed to replace the comparison between two boundaries with a comparison between a boundary and a point, at the cost of having some extra monotonicity requirement on the nonlinearities. These ideas were further developed, also in the case of ODEs, by Cabada and co-authors~\cite{acjcgi1, CCI2}. 

Here we apply a refinement of the results in~\cite{CCI2} to the elliptic system 
\begin{equation}
  \label{ellbvp-intro}
 \left\{
\begin{array}{ll}
   Lu_i(x)=\lambda_i f_i(x,u(x)), &  x\in \Omega,\quad i=1,2,\ldots,n, \\
 Bu_i(x)=0, & x\in \partial \Omega,\quad i=1,2,\ldots,n,
\end{array}
\right.
\end{equation}
where
$f_i\in C\bigl(\bar{\Omega}\times \prod_{j=1}^n[0,\rho_j]\bigr)$ and $\rho_i>0$ for $i=1,2,\ldots,n$. Our approach relies on the use of  the monotone iterative method jointly with the fixed point index. We stress that Cheng and Zhang~\cite{chzh} combined the method of sub-supersolutions with the Leray-Schauder degree, in the case of elliptic systems with Dirichlet BCs. 
Here we prove, under natural assumptions on the nonlinearities $f_i$, that nonzero positive weak solutions of the BVP~\eqref{ellbvp-intro} exist for certain values of the parameters $\lambda_i$. We also present some numerical examples to illustrate our theory.  As far as we know, our results are new and complement the results of~\cite{chzh, lan1}. 

\section{Preliminaries}
A subset $P$ of a
real Banach space $X$ is a {\it cone} if it is closed, $P+P\subset P$,
  $\lambda P\subset P$ for all $\lambda\ge 0$ and  $P\cap(-P)=\{\theta\}$. A cone $P$ defines the partial ordering in
$X$ given by 
$$ x\le y \quad \mbox{if and only if $y-x\in P$}.$$
 For $x,\ y \in X$,
with $x\le y$, we define the ordered interval \[[x,y]=\{z\in X : x\le z\le y\}.\]

The cone $P$ is {\it normal} if there exists $d>0$ such that for all $x, y\in X$ with $\theta \le x\le y$ then
$\|x\|\le d \|y\| $. The cone $P$ is {\it total} if $X=\overline{P-P}$. We remark that every cone $P$ in a normed space has the {\it Archimedean property}, that is, $n x\le y$ for all $n\in \N$ and some $y\in X$ implies $x\le \theta$. In the sequel, with abuse of notation, we will use the same symbol ``$\ge$" for the different cones appearing in the paper since the context will avoid confusion.

We denote the closed ball of center $x_0\in X$ and radius $r>0$ as
\[B[x_0,r]=\{x\in X : \|x-x_0\|\le r \},\] and the intersection of the cone with the open ball centered at the origin and radius $r>0$ as
$$P_r=P \cap \{x\in X : \|x\|<r\}.$$
In the sequel the closure and the boundary of subsets of $P$ are understood to be relative to $P$.
An operator $T:D\subseteq X\to X$ is said to be monotone non-decreasing if for every $x,y\in D$ with $x\leq y$ we have $Tx\leq Ty$.
We recall the following theorem, known as the monotone iterative method (see, for example,~\cite{Amann-rev,zeidler}). 

\begin{thm} \label{mim} Let $X$ be a real Banach space with normal order cone
$P$. Suppose that there exist $\alpha\le \beta$  such that $T\colon [\alpha,\beta]\subset P\to  P$ is a completely continuous monotone non-decreasing operator with $\alpha \le T \alpha $  and $T\beta \le \beta$. Then $T$ has a fixed point and the iterative sequence  $\alpha_{n+1}= T \alpha_n$, with $\alpha_0=\alpha$, converges to the smallest fixed point of $T$ in $[\alpha,\beta]$, and the sequence $\beta_{n+1} = T \beta_n$, with $\beta_0=\beta$, converges to the greatest fixed point of $T$ in $[\alpha,\beta]$. 
\end{thm}

In the following Proposition we recall the main properties of the fixed point index of a completely continuous operator relative to a cone, for more details see~\cite{Amann-rev, guolak}.  

\begin{pro}\label{propindex} Let $D$ be an open bounded set of $X$ with $0\in D_{P}$ and
$\overline{D}_{P}\ne P$, where $D_{P}=D\cap P$. 
Assume that $T:\overline{D}_{P}\to P$ is a completely continuous operator such that
$x\neq Tx$ for $x\in \partial D_{P}$. Then the fixed point index
 $i_{P}(T, D_{P})$ has the following properties:
 
\begin{itemize}

\item[$(i)$] If there exists $e\in P\setminus \{0\}$
such that $x\neq Tx+\lambda e$ for all $x\in \partial D_P$ and all
$\lambda>0$, then $i_{P}(T, D_{P})=0$.

\item[] For example $(i)$ holds if $Tx\not\leq x$ for $x\in
\partial D_P$.

\item[$(ii)$] If $\|Tx\|\ge \|x\|$ for $x\in
\partial D_P$, then $i_{P}(T, D_{P})=0$.

\item[$(iii)$] If $Tx \neq \lambda x$ for all $x\in
\partial D_P$ and all $\lambda > 1$, then $i_{P}(T, D_{P})=1$.

\item[] For example $(iii)$ holds if either $Tx\not\geq x$ for $x\in
\partial D_P$ or $\|Tx\|\le \|x\|$ for $x\in
\partial D_P$.

\item[(iv)] Let $D^{1}$ be open bounded in $X$ such that
$\overline{D^{1}_{P}}\subset D_P$. If $i_{P}(T, D_{P})=1$ and $i_{P}(T,
D_{P}^{1})=0$, then $T$ has a fixed point in $D_{P}\setminus
\overline{D_{P}^{1}}$. The same holds if 
$i_{P}(T, D_{P})=0$ and $i_{P}(T, D_{P}^{1})=1$.
\end{itemize}
\end{pro}

The following result, ensuring the existence of a non-trivial fixed point, is a modification of \cite[Theorem 2.3]{CCI2}. The short proof is presented here for completeness.
\begin{thm} \label{thCCI} Let $X$ be a real Banach space, $P$ a normal cone
with normal constant $d\ge 1$ and nonempty interior (i.e. solid) and for $\rho>0$ let $T\colon
\overline{P_{\rho}}\to P$ be a completely continuous operator. Assume that
\begin{itemize}
\item[(1)]  there exist $\beta \in \overline{P_{\rho/d}}$, with $T\beta\le \beta$, and $0<{R}<\rho$ such that $B[\beta,{R}]\subset P$,

\item[(2)] the map $T$ is non-decreasing in the set
\[ \tilde{\mathcal{P}}=\Bigl\{x\in P : x\le \beta \quad \mbox{and} \quad 
\frac{{R}}{d}\le \|x\|
\Bigr\},
\]{}
\item[(3)] there exists a (relatively) open bounded set $V\subset P_{\rho}$ such that $\theta\in V$, $i_P(T,V)=0$ and either
$\overline{P_{R}}\subset V$ or $\overline{V}\subset P_{R}$.
 \end{itemize}
Then the map $T$ has at least one non-zero fixed point $x_1$ in $P$,
 that
$$ \mbox{either belongs  to $\tilde{\mathcal{P}}$ or belongs to} 
\left\{ \begin{array}{ll} V\setminus \overline{P_{R}},  & \mbox{in case $\overline{P_{R}}\subset V$,} \\ P_{R}\setminus \overline{V},  & \mbox{in case $\overline{V}\subset P_{R}$.}\end{array}\right.
$$
\end{thm}

\begin{proof}  Since $B[\beta,{R}]\subset P$ we have that if $x\in
P$ with $\|x\|={R}$ then  $x\leq \beta$.

Suppose first that there exists $\alpha\in P_{\rho}$ with $\|\alpha
\|={R}$ and $T\alpha \geq \alpha$.  Then $\alpha\le \beta$ and due to the normality of the cone $P$ we have that 
$$ \alpha \le x \implies R=\| \alpha \|  \le d \|x\|.$$
So $[\alpha,
\beta]\subset \mathcal{\tilde P}$ and from (2) it follows that $T$ is non-decreasing on
$[\alpha, \beta]$ .
Then we can apply Theorem~\ref{mim} to ensure the existence of
a fixed point of $T$ on $[\alpha, \beta]$, which, in
particular, is a non-trivial fixed point.

Now suppose that such $\alpha$ does not exist and that $T$ is fixed point free on $\partial P_R$ (if not, we would have the desired non-trivial fixed point). Thus 
$Tx \ngeq x$ for all $x\in P_{\rho}$ with $\|x\|={R}$, which by Proposition~\ref{propindex}, $(iii)$ implies that $i_P(T,P_{R})=1$. Since, by assumption, $i_P(T,V)=0$ we get the existence of  a 
fixed point $x_1$ belonging to the set $V\setminus \overline{P_{R}}$ (when $\overline{P_{R}}\subset V$) or to the 
$P_{R}\setminus \overline{V}$  (when $\overline{V}\subset P_{R}$). In both cases $x_1$ would be a non-trivial fixed point of $T$.
\end{proof}

\begin{rem} In Theorem \ref{thCCI}, (3) with the condition $i_P(T,V)=0$, due to the construction of the fixed point index, we implicitly assume that $T$ is fixed point free on $\partial V$. On the other hand, notice that  Condition (i) in Proposition~\ref{propindex} provides a rather general condition implying index zero and, typically in applications, one would try $e\equiv 1$ or $e$ a positive eigenfunction related to a suitable compact linear operator, see~\cite{Amann-rev}.   We develop further this idea in the following section.
\end{rem}
 
\section{Sufficient conditions for index zero}

In this Section we deal with an abstract Hammerstein fixed point equation in a product space. For $i=1,2,\ldots,n$ let $X_i$ a Banach space ordered by a normal cone $P_i$ and consider the product Banach space $X=\prod_{i=1}^n X_i$ with the maximum norm and ordered by the cone $P=\prod_{i=1}^n P_i$. Let us consider the operators $F_i:D\subset X\to X_i$ and $K_i:X_i\to X_i$ and the fixed point equation
\begin{equation}\label{eqfixpoint}   
u=Tu:=(K_1 F_1 u,K_2 F_2 u,\ldots, K_n F_n u) \quad \mbox{for $u\in D\subset X$.}
\end{equation}

By exploiting the structure of the operator $T$ in the equation \eqref{eqfixpoint} we are able to prove the following result which generalizes \cite[Lemma 2.2]{lan2} in two ways: firstly, the inequality in our condition (H1) is asked only for one component $F_i$ and, secondly, we do not need to add a positive $\varepsilon$ to its right-hand side.  Thus, the applicability of our result is wider and it includes for instance  \cite[Theorem 4.3]{jw-tmna}  which is not covered by \cite[Lemma 2.2]{lan2}. 

\begin{thm}\label{thindex0-2}  Assume the following conditions:
\begin{itemize}
\item[(H0)] For each $i=1,2,\ldots,n$ the operator $K_i$ is linear, compact, $K_i(P_i)\subset P_i$ and there exist $ \alpha_i>0$ and $e_i\in P_i\setminus \{0\}$ such that $K_i e_i\ge \alpha_i e_i$.
\item[(H1)] There exists $\rho_0>0$ such that $\bar{P}_{\rho_0}\subset D$, for each $i=1,2,\ldots,n$ the operator $F_i$ is continuous and bounded on $\bar{P}_{\rho_0}$, $F_i(\bar{P}_{\rho_0})\subset P_i$, 
and there exist $i_0\in \{1,2,\ldots,n\}$ such that
$$F_{i_0}(u)\ge \frac{1}{\alpha_{i_0}} u_{i_0} \quad \mbox{for $u\in \partial P_{\rho_0}$.}$$
\end{itemize}
If $u \not=Tu$ for $u\in \partial P_{\rho_0}$ then $i_P(T, P_{\rho_0})=0$.

\end{thm}
\begin{proof}
Defining $e=(e_1,e_2,\ldots,e_n)$ by Proposition~\ref{propindex} it is enough to prove that
\begin{equation}\label{eqindex0-1} u\neq Tu+\mu e \quad \mbox{for all $u\in \partial P_{\rho_0}$ and all
$\mu >0$.}
\end{equation}
If not, there exists $u\in \partial P_{\rho_0}$ and $\mu >0$ such that
\begin{equation}\label{eqindex0-2} u= Tu+\mu e. 
\end{equation}
Then $u\ge \mu e$ and in particular $u_{i_0}\ge \mu e_{i_0}$. Now by (H0) and (H1) we have
\begin{multline*}
u_{i_0}=K_{i_0} F_{i_0} u+\mu e_{i_0}\ge K_{i_0} \left(\frac{1}{\alpha_{i_0}} u_{i_0}\right) +\mu e_{i_0}\ge K_{i_0}\left(\frac{1}{\alpha_{i_0}}  (\mu e_{i_0})\right) +\mu e_{i_0}\\
=\frac{1}{\alpha_{i_0}}  \mu K_{i_0}(e_{i_0}) +\mu e_{i_0}\ge \frac{1}{\alpha_{i_0}}  \mu \alpha_{i_0} e_{i_0} +\mu e_{i_0}=2(\mu e_{i_0}).
\end{multline*}
Then $u_{i_0}\ge 2 (\mu e_{i_0})$ and by induction it can be proven that actually $u_{i_0}\ge n (\mu e_{i_0})$ for all $n\in \N$. By the Archimedian property we obtain $\mu e_{i_0} \le 0$, a contradiction.  
\end{proof}

\begin{rem}\label{remlbsr} Notice that if (H0) holds then $\alpha_{i_0}$ is a lower bound for the spectral radius $r(K_{i_0})$, that is $r(K_{i_0})\ge \alpha_{i_0}>0$ (see \cite[Theorem 5.4]{kras} or \cite[Theorem 2.7]{jw-tmna}). So the best value for $\alpha_{i_0}$ in order to make (H1) as weak as possible would be $r(K_{i_0})$, and in fact we can choose $\alpha_{i_0}=r(K_{i_0})$  in (H1)  if the assumptions of the celebrated Krein-Rutman theorem hold.
\end{rem}

\begin{cor}\label{corindex0-3}  Suppose that (H0) holds and also 
\begin{itemize}
\item[(H1')]  There exists $\rho_0>0$ such that $\bar{P}_{\rho_0}\subset D$,
for each $i=1,2,\ldots,n$ the operator $F_i$ is continuous, bounded on $\bar{P}_{\rho_0}$, $F_i(\bar{P}_{\rho_0})\subset P_i$,
 and there exist $i_0\in \{1,2,\ldots,n\}$ such that $P_{i_0}$ is a total cone and
$$F_{i_0}(u)\ge \frac{1}{r(K_{i_0})} u_{i_0} \quad \mbox{for $u\in \partial P_{\rho_0}$.}$$
\end{itemize}
If  $u \not=Tu$ for $u\in \partial P_{\rho_0}$ then $i_P(T, P_{\rho_0})=0$.

\end{cor}
\begin{proof}
As explained in Remark \ref{remlbsr} we have $r(K_{i_0})\ge \alpha_{i_0}>0$ so the Krein-Rutman theorem  applies (see \cite[Theorem 7.26]{zeidler}) and then $r(K_{i_0})$ is an eigenvalue with an eigenvector $\varphi_{i_0} \in P_{i_0}\setminus \{0\}$, that is $K_{i_0}\varphi_{i_0}=r(K_{i_0})\varphi_{i_0}$. So (H0) also holds taking $\alpha_{i_0}=r(K_{i_0})$ which together with (H1') give us the desired conclusion applying Lemma \ref{thindex0-2}.
\end{proof}

\section{A fixed point formulation for elliptic systems}

\subsection{The assumptions}

We shall deal with the following system of second order elliptic BVPs
\begin{equation}
  \label{eqellipticsystem}
 \left\{
\begin{array}{ll}
   Lu_i(x)=\lambda_i f_i(x,u(x)), &  x\in \Omega,\quad i=1,2,\ldots,n, \\
 Bu_i(x)=0, & x\in \partial \Omega,\quad i=1,2,\ldots,n,
\end{array}
\right.
\end{equation}
where $\Omega\subset \R^m$, $m\ge 2$, is a suitable bounded domain,  $L$ is a strongly uniformly elliptic operator, $u(x)=(u_1(x),u_2(x),\ldots,u_n(x))$, $\lambda_i>0$, $f_i\in C\bigl(\bar{\Omega}\times \prod_{j=1}^n[0,\rho_j]\bigr)$ and $\rho_i>0$ for $i=1,2,\ldots,n$  and $B$ is a first order boundary operator. More precisely, we shall assume the conditions in \cite[Section~4 of Chapter~1]{Amann-rev} (see also \cite{lan1,lan2}), namely:

\begin{enumerate}

\item We suppose that  $\Omega\subset \R^m$, $m\ge 2$, is a bounded domain such that its boundary $\partial \Omega$ is an $(m-1)-$dimensional $C^{2,\mu}-$manifold for some $\mu\in (0,1)$, such that $\Omega$ lies locally on one side of $\partial \Omega$ (see \cite[Section 6.2]{zeidler} for more details).
\item $L$  is the second order elliptic operator given by
\begin{equation}\label{eqL} Lz(x)=-\sum_{i,j=1}^m a_{ij}(x)\frac{\partial^2 z}{\partial x_i \partial x_j}(x)+\sum_{i=1}^m a_{i}(x) \frac{\partial z}{\partial x_i} (x)+a(x)z(x), \quad \mbox{for $x\in \Omega$,}\end{equation}
where $a_{ij},a_{i},a\in C^{\mu}(\overline{\Omega})$ for $i,j=1,2,\ldots,m$, $a(x)\ge 0$ on $\bar{\Omega}$, $a_{ij}(x)=a_{ji}(x)$ on $\bar{\Omega}$ for $i,j=1,2,\ldots,m$.  Moreover $L$ is strongly uniformly elliptic, that is, there exists $\mu_0>0$ such that 
$$\sum_{i,j=1}^m a_{ij}(x)\xi_i \xi_j\ge \mu_0 \|\xi\|^2 \quad \mbox{for $x\in \Omega$ and $\xi=(\xi_1,\xi_2,\ldots,\xi_m)\in\R^m$.}$$

\item $B$ is the boundary operator given by
$$Bz(x)=b(x)z(x)+\delta \frac{\partial z}{\partial v}(x) \quad \mbox{for $x\in\partial \Omega$},$$
where $v$ is an outward pointing and nowhere tangent vector field on $\partial \Omega$ of class $C^{1,\mu}$ (not necessarily a unit vector field),  $\frac{\partial z}{\partial v}$ is the directional derivative of $z$ with respect to $v$,  $b:\partial \Omega \to \R$ is of class $C^{1,\mu}$ and moreover one of the following conditions holds:
\begin{enumerate}
\item $\delta =0$ and $b(x)\equiv 1$ (Dirichlet boundary operator).
\item $\delta =1$, $b(x)\equiv 0$ and $a(x)\not\equiv 0$ (Neumann boundary operator).
\item $\delta =1$, $b(x)\ge 0$ and $b(x)\not\equiv 0$ (Regular oblique derivative boundary operator).
\end{enumerate}
\end{enumerate}

\subsection{The solution operator} It is well-known, since the pioneering work by Schauder,  that under the previous conditions for each $g\in C^{\mu}(\bar{\Omega})$, the boundary value problem 
 \begin{equation}
  \label{eqelliptic}
 \left\{
\begin{array}{ll}
   Lz(x)=g(x), &  x\in \Omega, \\
 Bz(x)=0, & x\in \partial \Omega,
\end{array}
\right.
\end{equation}
has a unique classical solution $z\in C^{2,\mu}(\bar{\Omega})$ and the solution operator $K:C^{\mu}(\bar{\Omega})\to C^{2,\mu}(\bar{\Omega})$ defined as $Kg=z$ is linear, continuous and positive (due to the maximum principle). It is well known that if $g\in C(\bar{\Omega})$ and $m\ge 2$ then problem \eqref{eqelliptic} need not have a classical solution $z\in C^{2}(\bar{\Omega})$ (see a counterexample in \cite[Problem 6.3]{zeidler}). For our purposes it will be fundamental to consider a generalized solution operator extending $K$ to $C(\bar{\Omega})$. 

\begin{pro}[Generalized solution operator]\label{proK1} The operator $K$ can be extended uniquely to a continuous, linear and compact operator $K:C(\bar{\Omega})\to C(\bar{\Omega})$ (that we will denote again by the same name). 
\end{pro}

Consider the cone of positive functions $P=C(\bar{\Omega},\R_+)$. The generalized solution operator $K$ is not only positive (i.e. $K(P)\subset P$), but also $e$-positive, as we stress in the following result, see \cite[Lemma 5.3]{amann-JFA}.
\begin{pro} \label{proK2} Let $e=K1\in C(\bar{\Omega},\R_+)\setminus \{0\}$. Then $K:C(\bar{\Omega})\to C(\bar{\Omega})$ is $e$-positive (and in particular positive), that is for each $g\in C(\bar{\Omega},\R_+)\setminus \{0\}$ there exist $\alpha_g>0$ and $\beta_g>0$ such that $\alpha_g e \le K g \le \beta_g e$. In particular, there exists $\alpha_e>0$ such that $Ke\ge \alpha_e$e.

\end{pro}

\subsection{The Nemytskii operator}

If $I=\prod_{j=1}^n I_j\subset \R^n$, where $I_j\subset \R$ is a closed interval with nonempty interior for $j=1,2,\ldots,n$, and $D\subset C(\bar{\Omega},\R^n)$  is a nonempty set we define 
$$D_I=\{u\in D: u(x)\in I \quad \mbox{for all $x\in \bar{\Omega}$}\}.$$

We utilize the space $C(\bar{\Omega},\R^n)$, endowed with the norm $\|u\|:=\displaystyle\max_{i=1,2,\ldots,n} \{\|u_i\|_{\infty}\}$, where $\|z\|_{\infty}=\displaystyle\max_{x\in \bar{\Omega}}|z(x)|$, and consider (with abuse of notation) the cone $P=C(\bar{\Omega},\R^n_+)$.

For a function $f:\bar{\Omega}\times I\to\R$ the Nemytskii (or superposition) operator $F$ is defined as $Fu(x):=f(x,u(x))$ for $u\in C(\bar{\Omega},\R^n)_I$ and  $x\in \bar{\Omega}$.
We recall some useful properties of the Nemytskii operator, see for example~\cite{Amann-rev}; here by \emph{non-decreasing} we mean non-decreasing with respect to the components.

\begin{pro}\label{proF} If $f\in C(\bar{\Omega}\times I)$ then the following claims hold:

\begin{enumerate} 
\item  The Nemytskii operator $F:C(\bar{\Omega},\R^n)_I \to C(\bar{\Omega})$ is continuous and bounded. 

\item If $f(x,u)\ge 0$ for all $(x,u)\in \bar{\Omega}\times I$ with $u\in\R^n_+$ then $F$ is positive, that is, $F(P_I)\subset P$. 

\item If for every $x\in \bar{\Omega}$ the function $f(x,\cdot):I\to \R$ is non-decreasing then $F$ is non-decreasing.
 
\item  If $f\in C^{\mu}(\bar{\Omega}\times I)$ and $u\in C^1(\bar{\Omega},\R^n)_I$ then $Fu\in C^{\mu}(\bar{\Omega})$.

\end{enumerate}
\end{pro}

\subsection{The fixed point formulation and weak solutions}

Now we can use the previous results in order to formulate the elliptic system~\eqref{eqellipticsystem} as a fixed point problem in the space of continuous functions. 

Let us define $I=\prod_{i=1}^n [0,\rho_i]$ and the operator $T:C(\bar{\Omega},\R^n)_I \to  C(\bar{\Omega},\R^n)$ as 
\begin{equation}
    \label{opT}
T(u):=(\lambda_1 K F_1(u),\lambda_2 K F_2(u),\ldots,\lambda_n K F_n(u)),
\end{equation}
where $F_i$ is the Nemytskii operator given by $F_i (u)(x)=f_i(x,u(x))$ for all $x\in \bar{\Omega}$.

\begin{defn} We shall say that $u\in C(\bar{\Omega},\R^n)_I $ is a {\it weak solution} of the problem  \eqref{eqellipticsystem} if and only if $u$ is a fixed point of operator $T$, that is, $u=Tu$. 
\end{defn}
Note that if $f$ is sufficiently smooth then, by regularity theory, a weak solution of~\eqref{eqellipticsystem} is also a classical one.  
\begin{pro} \label{proT}  Suppose that $f_i\in C(\bar{\Omega}\times I)$ for all $i=1,2,\ldots,n$. Then the following claims hold:

\begin{enumerate}

\item $T:C(\bar{\Omega},\R^n)_I \to  C(\bar{\Omega},\R^n)$ is a completely continuous operator.

\item If moreover $f_i \ge 0$ for all $i=1,2,\ldots,n$, then $T(P_I)\subset P$.

\item If $f_i\in C^{\mu}(\bar{\Omega}\times I)$ for all $i=1,2,\ldots,n$, then $u\in C^{2,\mu}(\bar{\Omega},\R^n)_I$  is a solution of \eqref{eqellipticsystem} if and only if  $u\in C(\bar{\Omega},\R^n)_I$ and $u=Tu$.  

\end{enumerate} 
\end{pro}

\section{Existence of positive solutions for elliptic systems}

The following is our main result. Remember that by Proposition \ref{proK2} and Remark \ref{remlbsr} we have that $r(K)>0$, where $r(K)$ is the spectral radius of $K$. We denote by $\mu_1=1/r(K)$.

\begin{thm}\label{thmain}  Assume that $f_i\in C(\bar{\Omega}\times \prod_{j=1}^n [0,\rho] )$ for all $i=1,2,\ldots,n$ and moreover:
\begin{itemize}

\item[(a)] For each $x\in \bar{\Omega}$ and for every $i=1,2,\ldots,n$, $f_i(x,\cdot):\prod_{j=1}^n [0,\rho] \to [0, +\infty)$ is non-decreasing.  

\item[(b)] There exist $\delta \in (0,+\infty)$, $i_0\in \{1,2,\ldots,n\}$ and $\rho_0 \in (0,\rho)$ such that $f_{i_0}(x,u)\ge \delta u_{i_0}$ for all $x\in \bar{\Omega}$ and $u\in\prod_{j=1}^n [0,\rho_0]$.

\item[(c)] There exist $(\beta_1, \beta_2,\ldots ,\beta_n)\in \prod_{j=1}^n (0,\rho]$ such that $m_j(\beta_1, \beta_2,\ldots ,\beta_n)>0$ for all $j=1,2,\ldots,n$, $j\not=i_0$, where
$
m_j(\beta_1, \beta_2,\ldots ,\beta_n)= \displaystyle\max_{x\in \bar{\Omega}}f_j(x,\beta_1, \beta_2,\ldots ,\beta_n)$ for all $j=1,2,\ldots,n$.

\end{itemize}

Then the system~\eqref{eqellipticsystem} has a nonzero weak positive solution $u$ such that $0<\|u\|\leq \rho$ provided that 
\begin{equation}
\label{eqlambda-1}
0< \lambda_j \leq  
\frac{\beta_j}{m_j(\beta_1, \beta_2,\ldots ,\beta_n)  \| K(1)\|_{\infty}} \quad \mbox{for all $j=1,2,\ldots,n$, $j\not=i_0$,}
\end{equation}
and
\begin{equation}
\label{eqlambda-2}
\frac{\mu_1}{\delta}< \lambda_{i_0} \leq  
\frac{\beta_{i_0}}{m_{i_0}(\beta_1, \beta_2,\ldots ,\beta_n)  \| K(1)\|_{\infty}}.
\end{equation}
\end{thm}

\begin{proof} We consider the cone $P=C(\bar{\Omega},\R^n_+)$, which is normal with constant $d=1$, in the space $C(\bar{\Omega},\R^n)$ with the maximum norm. $T$ maps $\overline{P_\rho} \rightarrow P$ and is completely continuous and nondecreasing by (a). 

Note that by (c) we have $\beta=(\beta_1, \beta_2,\ldots ,\beta_n)\in \overline{P_\rho}$, $\beta$ lies in the interior of $P$ and furthermore, for every $i=1,2,\ldots,n$,
\begin{multline}
\lambda_i KF_i(\beta_1,\beta_2,\ldots ,\beta_n)(x) \leq \| \lambda_i KF_i(\beta_1,\beta_2,\ldots ,\beta_n)\|_{\infty}\\
\leq \|\lambda_i K(m_i(\beta_1, \beta_2,\ldots ,\beta_n))\|_{\infty} \leq \lambda_i m_i(\beta_1, \beta_2,\ldots ,\beta_n)  \| K(1)\|_{\infty}\leq \beta_i.
\end{multline}%

Finally, from Corollary \ref{corindex0-3} and (b) we obtain that either $T$ has a fixed point in $\partial P_{\rho_0}$ or either 
$i_P(T, P_{\rho_0})=0$ for $\lambda_1>\mu_1/\delta$. Thus, in the second case we achieve also a nonzero fixed point of $T$ by Theorem \ref{thCCI}. 
\end{proof}

\begin{rem}\label{rem-reg} Notice that by Proposition \ref{proT}, part (3), if $f_i\in C^{\mu}(\bar{\Omega}\times I)$ for all $i=1,2,\ldots,n$, then any weak solution belongs to $C^{2,\mu}(\bar{\Omega},\R^n)_I$ and it is also a classical solution of \eqref{eqellipticsystem}.

\end{rem}

As an application, we present an existence result for a system with Dirichlet boundary conditions and the Laplacian operator.
\begin{cor}
Take $\Omega=\{(x_1,x_2)\in \mathbb{R}^2 : x_1^2+x_2^2<1\}$ and consider the system
\begin{equation} \label{example}
\left\{
\begin{array}{ll}
-\Delta u_1=\lambda_1 (|(u_1,u_2)|^\frac{1}{2} + \tan  |(u_1,u_2)| ),& \text{in }\Omega , \\
-\Delta u_2=\lambda_2 |(u_1,u_2)|^2, & \text{in }\Omega , \\
u_1,u_2=0, & \text{on }\partial \Omega ,%
\end{array}%
\right. 
\end{equation}%
where $|(u_1,u_2)|:=\max \{ |u_1|,|u_2|\}$. 

If we define $\rho=\frac{15 }{64}\pi$ then the system~\eqref{example} has a nontrivial positive solution $(u_1,u_2)\in C^{2,\frac1 2}(\bar{\Omega},\R^2_+)$ such that $0<\|(u_1,u_2)\|\le \frac{15 }{64}\pi$ for every
$$
0< \lambda_1 \leq  
\frac{4\rho}{m_1(\rho, \rho) } \approx 1.669,\ 0< \lambda_2 \leq  
\frac{4\rho}{m_2(\rho, \rho)  } \approx 5.432.
$$
\end{cor}

\begin{proof} Let $f_1,f_2:\bar{\Omega}\times[0,\rho]\times [0,\rho] \to \R$ be given by
$$f_1(x_1,x_2,u_1,u_2)=|(u_1,u_2)|^\frac{1}{2} + \tan  |(u_1,u_2)|  \quad \mbox{and} \quad f_2(x_1,x_2,u_1,u_2)= |(u_1,u_2)|^2.$$
 Note that $f_1,f_2\in C^{1/2}(\bar{\Omega}\times [0,\rho]\times [0,\rho])$, are non-decreasing and, given $\delta>0$, $f_1$ satisfies condition (b) in Theorem \ref{thmain} for $\rho_0$ sufficiently small, due to the behaviour near the origin.
By direct calculation, $K(1)=\frac{1}{4}(1-x_1^2-x_2^2)$ so $\|K(1)\|=\frac{1}{4}$ and 
fixing $\beta_1=\beta_2=\rho=\frac{15 }{64}\pi$ the result follows from Theorem \ref{thmain} and Remark \ref{rem-reg}. \end{proof}

In the case of a single equation
\begin{equation}
  \label{eqlaplacian}
 \left\{
\begin{array}{ll}
  L z=\lambda f(x,z), & x\in \Omega, \\
 Bz=0, & x\in \partial \Omega,
\end{array}
\right.
\end{equation}
where $\lambda>0$ and $f:\bar{\Omega}\times [0,\rho]\to \R$, a more precise result than Theorem \ref{thmain} can be obtained.

\begin{thm} \label{thm-single} Assume that $f\in C(\bar{\Omega}\times [0,\rho])$ and moreover:
\begin{itemize}

\item[i)] For each $x\in \bar{\Omega}$ the function $f(x,\cdot):[0,\rho] \to [0, +\infty)$ is non-decreasing.

\item[ii)] There exist $\delta \in (0,+\infty)$ and $\rho_0 \in (0,\rho)$ such that $f(x,z)\ge \delta z$ for all $x\in \bar{\Omega}$ and $z\in [0,\rho_0]$.

\end{itemize}

Then the BVP~\eqref{eqlaplacian} admits a non-zero weak positive solution $z$ such that $0<\|z\|_{\infty}<\rho$ provided that 
\begin{equation}
\label{eqlambda}
\frac{\mu_1}{\delta}\leq \lambda < \sup_{0<s\le \rho}  \frac{s}{M(s)\|K(1)\|_{\infty}},
\end{equation}
where $M(s)= \displaystyle\max_{x\in \bar{\Omega}}
f(x,s)>0$ for all $s\in (0,\rho]$.
\end{thm}

\begin{proof}
By our assumptions, $T:P_{\rho}\to P$ is completely continuous and non-decreasing. Now, let $0<\beta\le \rho$ such that  $\lambda <  \displaystyle\frac{\beta}{M(\beta) \|K(1)\|_{\infty}}$. Then we have
$$T\beta = \|T\beta \|_{\infty}=\| K(\lambda F \beta)\|_{\infty} \leq\| \lambda  M(\beta) K(1) \|_{\infty}= \lambda  M(\beta) \|K(1)\|_{\infty}<\beta.$$

Finally, if we define $V=P_{\rho_0}$ then by Corollary \ref{corindex0-3} we have that either $T$ has a fixed point in $\partial V$ (which will be a non-zero positive solution) or either $i_P(T,V)=0$. In this last case Theorem \ref{thCCI} applies and we obtain the existence of a non-zero weak positive solution.
\end{proof}

\begin{cor}\label{exlaplacian} The problem
\begin{equation}
  \label{eqexample}
 \left\{
\begin{array}{ll}
  -\Delta z=\lambda (\sqrt{z}+\tan(z)), & x\in \Omega=\{(x_1,x_2)\in \R^2 : x_1^2+x_2^2<1\}, \\
 z(x)=0, & x\in \partial \Omega,
\end{array}
\right.
\end{equation}
has a non-zero positive solution $z$ such that $0<\|z\|_{\infty}<\frac{\pi}{2}$ provided that 
\begin{equation}
\label{ eqlambdaexample}
0 <\lambda < \sup_{ 0<s<\frac \pi 2} \frac{4 s}{\sqrt{s}+\tan(s)} \approx 1.66924.
\end{equation}
\end{cor}

\begin{proof}  Firstly, take into account that $K(1)(x_1,x_2)=\displaystyle \frac{1}{4}(1-x_1^2-x_2^2)$  for $\Omega=\{(x_1,x_2)\in \R^2 : x_1^2+x_2^2<1\}$ so  $\|K(1)\|=\frac{1}{4}$. Therefore, if $\lambda>0$ satisfies \eqref{ eqlambdaexample} we can choose $\delta>0$ and $0<\rho<\frac{\pi}{2}$ such that \eqref{eqlambda} is also satisfied. Since $f(x_1,x_2,z)=\sqrt{z}+\tan(z)$ is non-decreasing in $[0,\rho]$ and the growth condition at $z=0$, we have that conditions i) and ii), for every $\delta>0$, in Theorem~\ref{thm-single} are satisfied, and then the existence of a non-zero weak positive solution follows taking $\beta=\rho$. Since  $f\in C^{\frac 1  2}(\bar{\Omega}\times [0,\rho])$ the solution is also a classical one (see Remark \ref{rem-reg}).
\end{proof}

 \begin{rem} Since $f(0)=0$ then clearly $u\equiv 0$ is a solution of \eqref{eqexample}. Example \ref{exlaplacian} shows a {\it non-trivial} positive solution that cannot be obtained from \cite[Theorem 9.4]{Amann-rev} and neither seems to be covered by the results in \cite[Section 1]{lions}.
\end{rem}
\section*{Acknowledgements}
The authors wish to thank Dr. Mateusz Maciejewski for the useful comments and the anonymous Referee for the constructive remarks.
J. A. Cid was supported by Xunta de Galicia (Spain), project EM2014/032.
G. Infante was partially supported by G.N.A.M.P.A. - INdAM (Italy).


\begin{thebibliography}{xxx}

\bibitem{alves-defig}
C. O. Alves and D. G. de Figueiredo,
Nonvariational elliptic systems,
\textit{Discrete Contin. Dyn. Syst.},
\textbf{8} (2002), 289--302.

\bibitem{amann-JFA} H. Amann, On the number of solutions of nonlinear
equations in ordered Banach spaces, {\it J. Functional Analysis},
{\bf 11} (1972), 346--384.

\bibitem{Amann-rev} H. Amann, 
Fixed point equations and nonlinear eigenvalue
problems in ordered Banach spaces, \textit{SIAM. Rev.}, \textbf{18} (1976),
620--709.

\bibitem{cabcid} A. Cabada and J. A. Cid, Existence of a non-zero fixed point for nondecreasing operators
via Krasnoselskii's fixed point theorem, 
\textit{Nonlinear Anal.},
\textbf{71} (2009), 2114--2118.

\bibitem{acjcgi1} A. Cabada, J. A. Cid and G. Infante, New criteria for the existence of non-trivial fixed points in cones, \textit{ Fixed Point Theory Appl.}, {\bf  2013:125} (2013). 

\bibitem{jc-df-fm}
J. A. Cid, D. Franco and F. Minh\'{o}s, 
Positive fixed points and fourth-order equations, \textit{Bull. Lond. Math. Soc.},
\textbf{41} (2009), 72--78.

\bibitem{CCI2} A. Cabada, J. A. Cid and G. Infante, A positive fixed point theorem with applications to systems of Hammerstein integral equations, \textit{Bound. Val. Probl.},
\textbf{2014:254}, (2014), 10 pp.

\bibitem{chzh} X. Cheng and Z. Zhang, Positive solutions for a class of multi-parameter elliptic systems, \textit{Nonlinear Anal. Real World Appl.},
\textbf{14} (2013), 1551--1562.

\bibitem{chzh07} 
X. Cheng and C. Zhong, Existence of three nontrivial solutions for an elliptic system, \textit{J. Math. Anal. Appl.},
\textbf{327} (2007), 1420--1430. 

\bibitem{cui3}
R. Cui, P.  Li, J. Shi and Y. Wang,
Existence, uniqueness and stability of positive solutions for a class of semilinear elliptic systems,
\textit{Topol. Methods Nonlinear Anal.},
\textbf{42} (2013), 91--104.

\bibitem{defigetal}
D. G. de Figueiredo, J. M. do {\'O} and B. Ruf,
Non-variational elliptic systems in dimension two: a priori bounds and existence of positive solutions, \textit{J. Fixed Point Theory Appl.},  \textbf{4} (2008), 77--96.

\bibitem{defig}
D. G. de Figueiredo, \textit{Semilinear elliptic systems: existence, multiplicity, symmetry of solutions}. In: Handbook of Differential Equations, Vol. 5: Stationary Partial Differential Equations, M. Chipot (ed.), Elsevier, 2008, 1--48.

\bibitem{df-gi-jp-prse}
D. Franco, G. Infante and J. Per\'an, A new criterion for the existence of multiple solutions in cones, \textit{Proc. Roy. Soc. Edinburgh Sect. A}, {\bf 142} (2012), 1043--1050.

\bibitem{guolak} D. Guo and V. Lakshmikantham,
\textit{Nonlinear problems in abstract cones}, Academic Press, Boston,
1988.

\bibitem{gimmrp}
G. Infante, M. Maciejewski and R. Precup,
A topological approach to the existence and multiplicity of positive solutions of $(p,q)$-Laplacian systems, \textit{Dyn. Partial Differ. Equ.}, \textbf{12} (2015), 193--215. 

\bibitem{kras} M. A. Krasnosel'skii, G. M. Vainikko, P. P. Zabreiko, Ya. B.  Rutitskii and V. Ya. Stetsenko, \textit{Approximate solution of operator equations}, Wolters-Noordhoff Publishing, Groningen, 1972. 

\bibitem{lan1} K. Q. Lan, Nonzero positive solutions of systems of elliptic boundary value problems, {\it Proc. Amer. Math. Soc.}, \textbf{139} (2011),  4343--4349.

\bibitem{lan2} K. Q. Lan, Existence of nonzero positive solutions of systems of second order elliptic boundary value problems \textit{J. Appl. Anal. Comput.}, \textbf{1} (2011), 21--31.

\bibitem{lan-zhang} K. Q. Lan and Z. Zhang, Nonzero positive weak solutions of systems of $p$-Laplace equations, \textit{J. 
Math. Anal. Appl.}, \textbf{394} (2012),  581--591.

\bibitem{lions} P. L. Lions, On the existence of positive solutions of semilinear elliptic equations, \textit{SIAM. Rev.}, \textbf{24} (1982), 441--467. 

\bibitem{ma2} R. Ma, R. Chen and Y. Lu, Positive solutions for a class of sublinear elliptic systems, \textit{Bound. Value Probl.},
\textbf{2014:28}, (2014), 15 pp.

\bibitem{persson} H. Persson, A fixed point theorem for monotone functions, \textit{Appl. Math. Lett.}, \textbf{19} (2006), 1207--1209.

\bibitem{bruf}
B. Ruf, \textit{Superlinear Elliptic Equations and Systems}. In: Handbook of Differential Equations, Vol. 5: Stationary Partial Differential Equations, M. Chipot (ed.), Elsevier, 2008, 211--276.

\bibitem{jw-tmna} J. R. L. Webb, A class of positive linear operators and applications to nonlinear boundary value problems, \textit{Topol. Methods Nonlinear Anal.}, \textbf{39} (2012), 221--242.

\bibitem{zeidler} E. Zeidler, {\it Nonlinear functional analysis and its applications.
 I. Fixed-point theorems}, Springer-Verlag, New York (1986).

\bibitem{zhch} Z. Zhang and X. Cheng, Existence of positive solutions for a semilinear elliptic system, \textit{Topol. Methods Nonlinear Anal.},
\textbf{37}  (2011), 103--116. 

\bibitem{zou}
H. Zou, A priori estimates for a semilinear elliptic system without variational structure and their applications, \textit{Math. Ann.},
\textbf{323} (2002) 713--735.

\end{thebibliography}
\end{document}